\documentclass[12pt]{amsart}
\usepackage{amsfonts,amssymb,amsrefs,comment,url}
\usepackage[usenames]{color}
\usepackage[all]{xy}
\usepackage{a4wide}
\begin{document}

\newcommand{\Ima}{\operatorname{Im}}
\newcommand{\Irr}{\operatorname{Irr}}
\newcommand{\sss}{\mathfrak{s}}
\newcommand{\Cusp}{\operatorname{Cusp}}
\newcommand{\Temp}{\operatorname{Temp}}
\newcommand{\Std}{\operatorname{Std}}
\newcommand{\Cyc}{\operatorname{Cyc}}
\newcommand{\JH}{\operatorname{JH}}
\newcommand{\Hom}{\operatorname{Hom}}
\newcommand{\Aut}{\operatorname{Aut}}
\newcommand{\Mult}{\operatorname{Mult}}
\newcommand{\std}[1]{\mathfrak{z}(#1)}
\newcommand{\N}{\mathbb{N}}
\newcommand{\C}{\mathbb{C}}
\newcommand{\R}{\mathbb{R}}
\newcommand{\Z}{\mathbb{Z}}
\newcommand{\MS}[1]{\N({#1})}
\newcommand{\id}{\operatorname{id}}
\newcommand{\End}{\operatorname{End}}
\newcommand{\Ker}{\operatorname{Ker}}
\newcommand{\sgn}{\operatorname{sgn}}
\newcommand{\Seg}{\operatorname{Seg}}
\newcommand{\Reps}{\operatorname{Rep}}
\newcommand{\GL}{\operatorname{GL}}
\newcommand{\supp}{\operatorname{supp}}
\newcommand{\m}{\mathfrak{m}}
\newcommand{\n}{\mathfrak{n}}
\newcommand{\soc}{\operatorname{soc}}
\newcommand{\indu}{\operatorname{Ind}}
\newcommand{\cosoc}{\operatorname{cosoc}}
\newcommand{\lshft}[1]{\overset{\leftarrow}{#1}}               
\newcommand{\rshft}[1]{\overset{\rightarrow}{#1}}              
\newcommand{\abs}[1]{\left|{#1}\right|}
 \newcommand{\Symm}{\mathcal{S}}

\newcommand{\rest}{\big|}

\newtheorem{theorem}{Theorem}[section]
\newtheorem{lemma}[theorem]{Lemma}
\newtheorem{problem}{Problem}
\newtheorem{proposition}[theorem]{Proposition}
\newtheorem{conjecture}[theorem]{Conjecture}
\newtheorem{observation}[theorem]{Observation}
\newtheorem{question}{Question}
\newtheorem{remark}[theorem]{Remark}
\newtheorem{definition}[theorem]{Definition}
\newtheorem{corollary}[theorem]{Corollary}
\newtheorem{example}[theorem]{Example}
\newtheorem{claim}[theorem]{``Claim''}
\newtheorem{assumption}[theorem]{Assumption}
\newtheorem{wassumption}[theorem]{Working Assumption}

\newcommand{\Alberto}[1]{{\color{red}{#1}}}
\newcommand{\Max}[1]{{\color{blue}{#1}}}

\title[Cyclic representations of general linear $p$-adic groups]
{Cyclic representations of general linear $p$-adic groups}
\author{Maxim Gurevich}
\address{Department of Mathematics, Technion - Israel Institute of Technology, Haifa, Israel, 3200003.}
\email{maxg@technion.ac.il}

\author{Alberto M\'inguez}
\address{Fakultät für Mathematik, University of Vienna, Oskar-Morgenstern-Platz 1, 1090 Wien (Austria).}
\email{alberto.minguez@univie.ac.at}

\date{\today}

\begin{abstract}
Let $\pi_1,\ldots,\pi_k$ be smooth irreducible representations of $p$-adic general linear groups. We prove that the parabolic induction product $\pi_1\times\cdots\times \pi_k$ has a unique irreducible quotient whose Langlands parameter is the sum of the parameters of all factors (cyclicity property), assuming that the same property holds for each of the products $\pi_i\times \pi_j$ ($i<j$), and that for all but at most two representations $\pi_i\times \pi_i$ remains irreducible (square-irreducibility property). Our technique applies the recently devised Kashi\-wara-Kim notion of a normal sequence of modules for quiver Hecke algebras.

Thus, a general cyclicity problem is reduced to the recent Lapid-M\'inguez conjectures on the maximal parabolic case.







%
%
%
\end{abstract}

\maketitle
%
%
\section{Introduction}

Parabolic induction is a pivotal part of the smooth representation theory of reductive $p$-adic groups. In this note, we consider the fundamental case of the group $G = \GL_n(F)$, for a non-Archimedean local field $F$. Each Levi subgroup $M<G$ is isomorphic to a product of the form $\GL_{n_1}(F)\times \cdots \times \GL_{n_k}(F)$, with $n_1+\ldots+n_k = n$. An irreducible representation of $M$ takes the form $\pi_1\otimes \cdots\otimes\pi_k$, where each $\pi_i$ is an irreducible $\GL_{n_i}(F)$-representation.

We view parabolic induction as a product operation -- the so-called Bernstein-Zelevinsky product:
\[
\pi_1\times \cdots \times \pi_k : = \indu(\pi_1\otimes\cdots\otimes \pi_k)\;.
\]

While such products are known to be of finite length, information of higher precision on the nature of the resulting product representation remains elusive. One desired goal of the theory would be a satisfactory understanding of the subrepresentations lattice of the induced representation. 

In parallel development similar questions have been explored in the setting of representations of quiver Hecke algebras \cite{MR3314831,MR3758148} and of quantum affine algebras \cite{MR3916087}. In fact, when fixing type $A$ Lie data, those settings are closely related to the context of our problems through various known categorial equivalences.

The  Zelevinsky classification (in its dual/Langlands form, see Section \ref{sec:seg}), is a bijection $\m \mapsto L(\m)$ between multisegments -- an essentially combinatorial object -- and irreducible representations of $G$. The representation $L(\m)$ is  defined as the unique quotient (the Langlands quotient) of the \textit{standard} representation attached to $\m$.

A basic property of this classification is that for any multisegments $\m_1, \dots, \m_k$ the representation $L(\m_1+\dots+\m_k)$ occurs with multiplicity one in the Jordan--H\"older sequence of the representation $L(\m_1) \times\dots\times L(\m_k)$. 

Therefore it appears natural to ask the following questions:
\begin{enumerate}
  \item When is $L(\m_1+\dots+\m_k)$ a quotient of $L(\m_1) \times\dots\times L(\m_k)$?
  \item When is it the unique irreducible quotient?
\end{enumerate}

This latter property is a particular case of a more general notion of cyclicity, which we import from the quantum affine setting \cite{MR3916087}. We call a (possibly reducible) representation $\pi$ of $G$ \textit{cyclic}, if $\pi$ can be realized as a quotient of a standard representation.

Lapid and the second-named author have stated a series of geometric conjectures \cite{LM-conj   } addressing the two properties from above, in the case of $k=2$.

Meanwhile, in the domain of quiver Hecke algebras, Kashiwara-Kim \cite{MR4016058} introduced the notion of a \textit{normal sequence} of simple modules. Translated to our setting, normal sequences would constitute sequences of certain representations $\pi_1,\ldots,\pi_k$ with the additional assumption that $\pi_i\times \pi_i$ are irreducible (a key condition coined as \textit{square-irreducible} \cite{MR3866895}, or \textit{real} \cite{MR3314831}). Normal sequences possess the property that the product $\pi_1\times \cdots \times \pi_k$ has a unique irreducible quotient.

The prominence of the square-irreducible property has been growing in the domain of $p$-adic groups as well. In fact, the conjectures of \cite{LM-conj} were motivated and proved in certain cases with the additional assumption that at least one the two representations is square-irreducible.

In this note we reduce the above-mentioned questions from the case of $k>2$ to the case of $k=2$, again with an assumption on square-irreducibility: We need to assume it on all factors, but at most two.


Our main theorem, Theorem \ref{thm:main-sqirr}, states that a relative cyclicity condition on a sequence of square-irreducible representations is sufficient to fulfil the requirement for a normal sequence, in the language of \cite{MR4016058}. In particular we obtain in Corollary \ref{cor:main},  a simple proof to the fact that under square-irreducibility assumptions, $\pi_1 \times \cdots\times \pi_k$ is irreducible, if and only if, $\pi_i\times \pi_j$, $i\neq j$ are irreducible.

An attempt to resolve similar questions in the absence of square-irreducibility assumptions is expected to require methods that are different from what is used in this note, and may follow the lines of the mentioned work of Hernandez \cite{MR3916087}. We leave this task for future efforts.

A particular application of that Theorem \ref{thm:main-sqirr} surfaced recently in the work of the first-named author \cite{me-restriction} in the proof of non-tempered Gan-Gross-Prasad branching laws \cite{ggp-non}. The techniques used in that work required some non-trivial categorical equivalences that translate $p$-adic problems into the quantum affine algebras realm. Since all products that appear in that context are built out of square-irreducible representations, our new ``native" $p$-adic proof considerably simplifies the proof of the criterion used in \cite{me-restriction}.

\subsection*{Acknowledgements}

We would like to thank David Hernandez, David Kazhdan, Erez Lapid and Marko Tadi\'c for useful discussions.
We also thank Guy Henniart for pointing out an embarrassing error in a previous version of this note.

The first-named author is partially supported by the Israel Science Foundation (grant No. 737/20). This work was begun while the first-named author was a research fellow at the National University of Singapore, supported by MOE Tier 2 grant 146-000-233-112. The first-named author would like to thank Wee Teck Gan and Lei Zhang for support during that fellowship, and the University of Vienna for their hospitality and support.

\section{Preliminaries}
\subsection{Representation theory of non-Archimedean $\GL_n$}

Let $F$ be a fixed non-Archime\-dean locally compact field of residue characteristic $p$.

For $n\in \mathbb{Z}_{\ge1}$, let $\Reps_n$ denote the abelian category of complex, smooth, finite length representations of the group $G_n:=\GL_n(F)$. Let $\Irr_n$ denote the set of isomorphism classes of irreducible elements in $\Reps_n$. Set  $\Reps=\oplus_{n\ge0}\Reps_n$ and $\Irr=\coprod_{n\ge0}\Irr_n$. Here we treat $\Reps_0$ as the trivial group.

For any $\pi\in \Reps_n$ and any $s\in \mathbb{C}$, we write $\pi(s)$ for the twist of $\pi$ by  $|{\rm det}|^s_{F}$, where  $|\ |_{F}$ is the normalized absolute value on $F$. We denote by $\pi^\vee$ the contragredient (smooth dual) representation of $\pi \in \Reps$.

We will sometimes refer to the underlying complex vector space of a representation $\pi\in \Reps$ by the same notation $\pi$.

Let $\mathcal{R}_n$ be Grothendieck group of $\Reps_n$. We denote by $\pi\mapsto [\pi]$ the natural semi-simplification map $\Reps_n \to \mathcal{R}_n$.
We write $\tau_1\leq \tau_2$ for elements $\tau_1,\tau_2\in \mathcal{R}_n$, if $\tau_2-\tau_1 = [\sigma]$, for a representation $\sigma \in \Reps_n$.

\subsection{Parabolic induction}

For $n\in \mathbb{Z}_{\ge1}$, let $\alpha = (n_1, \ldots, n_r)$ be a composition of $n$. Let $M_\alpha$ be the subgroup of $G_n$ isomorphic to $G_{n_1} \times \cdots \times G_{n_r}$ consisting of matrices which are diagonal by blocks of size $n_1, \ldots, n_r$. Let $P_\alpha$ the subgroup of $G_n$ generated by $M_\alpha$ and the upper
unitriangular matrices. A standard parabolic subgroup of $G_n$ is a subgroup of the form $P_\alpha$ and its standard Levi factor is $M_\alpha$.


For given representations $\pi_i\in \Reps_{n_i}$, $i=1,\ldots,r$, we write
\[
\pi_1\times\cdots\times \pi_r\in \Reps_{n_1+ \ldots +n_r}
\]
to be the representation obtained by normalized parabolic induction (through $P_\alpha$) of the representation $\pi_1\otimes\cdots\otimes \pi_r$ of the group $M_\alpha$, where $\alpha = (n_1,\ldots,n_r)$.

We write $\Cusp$ the subset of $\Irr$ made of cuspidal representations, that is irreducible representations which are not quotients of proper induced representations.

Recall that $\times$ gives an associative product structure on the group $\oplus_{n \ge0} \mathcal{R}_n$. Moreover,
the product is commutative, in the sense that $[\pi_1\times\pi_2] = [\pi_2 \times \pi_1]$, for all $\pi_1, \pi_2 \in \Reps$. In particular, if $\pi_1\times\pi_2 \in \Irr$ then $\pi_1\times\pi_2 \simeq \pi_2 \times \pi_1$.

On the categorical level, the product functor $(\pi_1,\ldots, \pi_1) \mapsto \pi_1\times \cdots\times \pi_r$ is exact in each variable. In particular, when $\tau$ is a subrepresentation of $\pi_1$, $\tau\times \pi_2$ is viewed naturally as a sub-representation of $\pi_1\times \pi_2$.

One symmetry of the categories involved in $\Reps$ is often referred to as the Gelfand-Kazhdan duality \cite{MR0404534}. We will state here one corollary of this symmetry which will be useful for our purposes.

\begin{proposition}\label{prop:GK}
For all $\pi_1,\ldots,\pi_t \in \Irr$, the socle (maximal semisimple subrepresentation) of $\pi_1\times \cdots\times \pi_t$ is isomorphic to the co-socle (maximal semisimple quotient) of $\pi_t\times \cdots\times \pi_1$ .


\end{proposition}

\subsection{Intertwining operators}
We recall here some well known facts about intertwining operators \cite[\textsection IV.1]{MR1989693}.
For $\pi_1,\pi_2\in \Reps$ and $s\in \mathbb{C}$, we write $M_{\pi_1,\pi_2}(s)$ for the standard intertwining operator
\[
M_{\pi_1,\pi_2}(s):\pi_1(s)\times\pi_2\rightarrow\pi_2\times\pi_1(s)\;.
\]

It depends on the choice of a Haar measure, but this dependence will not play a role in our discussion.


For representations $\pi_1,\pi_2\in \Reps$, let $d_{\pi_1,\pi_2}\ge0$ be the order of the pole of $M_{\pi_1,\pi_2}(s)$ at $s=0$ and let
\[
R_{\pi_1,\pi_2}=\lim_{s\rightarrow0}s^{d_{\pi_1,\pi_2}}M_{\pi_1,\pi_2}(s).
\]
Thus, $R_{\pi_1,\pi_2}$ is a non-zero intertwining operator from $\pi_1\times\pi_2$ to $\pi_2\times\pi_1$.

The following properties are known to hold \cite[Lemma 2.3]{MR3866895}\footnote{In parts (6) and (7) of Lemma 2.3 \cite{MR3866895}, one needs to require that $\pi_1$ and $ \pi_2$ are irreducible.}.
\begin{lemma} \label{lem:intert}
Let $0\ne\pi_i\in\Reps(G_{n_i})$, $i=1,2,3$ and let $\tau$ de a subrepresentation of $\pi_1$. Then
\begin{enumerate}
\item\label{part: nonzeroR} The operator $R_{\pi_1,\pi_2}$ restricts to an intertwining operator $\tau\times\pi_2\rightarrow\pi_2\times\tau$. More precisely,
\[
R_{\pi_1,\pi_2}\rest_{\tau\times\pi_2}=\begin{cases}R_{\tau,\pi_2}&\text{if }\tau\ne0\text{ and }d_{\pi_1,\pi_2}=d_{\tau,\pi_2},\\
0&\text{otherwise.}\end{cases}
\]
\item\label{part: ifisom} The inequality $d_{\pi_1\times\pi_2,\pi_3}\le d_{\pi_1,\pi_3}+d_{\pi_2,\pi_3}$ holds, together with the equality
\[
(R_{\pi_1,\pi_3}\times\id_{\pi_2})\circ(\id_{\pi_1}\times R_{\pi_2,\pi_3})=
\begin{cases}R_{\pi_1\times\pi_2,\pi_3}&\text{if }d_{\pi_1\times\pi_2,\pi_3}=d_{\pi_1,\pi_3}+d_{\pi_2,\pi_3},\\0&\text{otherwise.}\end{cases}
\]
Moreover, $d_{\pi_1\times\pi_2,\pi_3}=d_{\pi_1,\pi_3}+d_{\pi_2,\pi_3}$ when at least one of $R_{\pi_i,\pi_3}$, $i=1,2$ is an isomorphism.
\end{enumerate}
\end{lemma}

\section{Cyclicity}\label{sect:cycl}
\subsection{The theory of segments}\label{sec:seg}
Let us briefly recall the Langlands-Zelevinsky theory \cite{MR584084} of the classification of the elements of $\Irr$.

A \emph{segment} is a nonempty finite set of the form
\[
[a,b]_\rho=\{\rho(a),\rho(a+1), \dots,\rho(b)\}
\]
where $\rho \in\Cusp$ and $a \leq b$ are some integers.
For any such $[a,b]_\rho$ we set $L([a,b]_\rho)\in \Irr$ the unique irreducible quotient of:
\[
\rho(a) \times \rho(a+1)\times\dots \times\rho(b).
\]
We say that $[a',b']_{\rho'}$ \emph{precedes} $[a,b]_{\rho}$ if $\rho(b) \notin [a',b']_{\rho'}$, $\rho'(b')  \in[a-1,b-1]_{\rho}$ and
$\rho'(a') \notin [a,b]_{\rho}$.

A \emph{multisegment} is a formal sum $\m= \Delta_1 +\dots+\Delta_N$ of segments. We denote $\Mult$ the set of multisegments. To each multisegment $\m$ we attach a representation $\Sigma(\m)\in \Reps$ in the following way. We choose a sequence of segments $\Delta_1, \dots, \Delta_N$ such that for all $i<j$, $\Delta_i$ does not precede $\Delta_j$  and $\m=\Delta_1 +\dots+\Delta_N$. The sequence $\Delta_1, \dots, \Delta_N$ always exists but is not defined uniquely. However the induced representation
$$\Sigma(\m):= L(\Delta_1)\times\dots\times L(\Delta_N)$$
is uniquely determined by $\m$ and is called the \emph{standard representation} attached to $\m$. We denote $\Std$ the set of standard representations.

For every $\m\in \Mult$, we denote by $L(\m)$ the unique irreducible quotient of $\Sigma(\m)$. The Zelevinsky classification (in its dual form) states that the map $\m \mapsto L(\m)$ is a bijection from $\Mult$ to $\Irr$.

We will also write $\Sigma(\pi):= \Sigma(\m)\in \Std$, when $\pi=L(\m)\in \Irr$.

The sum of multisegments induces a additive structures $\ast$ and $\diamond$ in $\Irr$ and $\Std$, respectively. Precisely, if $\m,\m'\in \Mult$, then $L(\m)\ast L(\m'):=L(\m+\m')$ and $\Sigma(\m)\diamond \Sigma(\m'):=\Sigma(\m+\m')$. We say that $L(\m)\preceq L(\m')$,
when $[L(\m)]\leq [\Sigma(\m')]$ (i.e. the isomorphism class of $L(\m)$ appears as a subquotient in $\Sigma(\m')$). This relation puts a partial order on $\Irr$ \cite[\textsection 7]{MR584084}.

\begin{remark}\label{rem:sum}
\begin{enumerate}
  \item As a consequence of the commutativity of the parabolic induction on the Grothendieck group level, it is easy to see that $[\Sigma_1\diamond \Sigma_2] = [\Sigma_1\times \Sigma_2]$, for all $\Sigma_1,\Sigma_2\in \Std$.

  \item\label{it:sum2}
It is known (see, for example, \cite[Theorem 2.6]{MR3573961}) that $\pi_1\ast \cdots \ast \pi_r$ appears with multiplicity $1$ in the Jordan-H\"{o}lder series of $\pi_1\times\cdots \times \pi_r$, for all $\pi_1,\ldots,\pi_r\in \Irr$. In particular, when $\pi_1\times\cdots \times \pi_r$ is irreducible, we necessarily have $ \pi_1\times\cdots \times \pi_r \cong \pi_1\ast\cdots \ast \pi_r$.

\end{enumerate}
\end{remark}

\subsection{Cyclic representations}
\begin{definition}
We say that a representation $\pi\in \Reps$ is \textit{cyclic}, if there is a standard representation $\Sigma\in \Std$ with a surjective homomorphism $\Sigma \to \pi$.
\end{definition}

Our choice of terminology for this definition is based on the analogous notion in the quantum affine setting \cite{MR3916087}, where cyclic representations are those that are generated by their highest weight vector. In both settings the notion differs from the purely algebraic notion of being generated by a  single vector.

Yet, let us mention that it is known (\cite[Proposition XI.2.6(4)]{MR1721403}) that for cyclic representations in our sense there is indeed a (choice of a) generating vector.

\begin{definition}
We say that a tuple of representation $\pi_1,\ldots,\pi_k \in \Reps$ satisfies $\Cyc(\pi_1,\ldots,\pi_k)$, if $\pi_1\times\cdots\times \pi_k$ is cyclic.
\end{definition}

All irreducible representations in $\Reps$ are cyclic by the Zelevinsky classification. In particular, if $\pi_1\times \cdots \times \pi_k$ is irreducible for $\pi_1,\ldots,\pi_k\in \Irr$, then $\Cyc(\pi_1,\ldots,\pi_k)$ holds. More explicitly, in this case we have a surjection
\[
\Sigma(\pi_1) \diamond\cdots \diamond \Sigma(\pi_k) \to \pi_1 \ast \cdots\ast \pi_k\cong \pi_1\times\cdots \times \pi_k\;.
\]

\begin{proposition}\label{prop:cyclic}
For representations $\pi_1 ,\ldots, \pi_k\in \Irr$, the following are equivalent:
\begin{enumerate}
  \item Property $\Cyc(\pi_1,\ldots,\pi_k)$ holds.\label{it:1}
  \item The product $\pi_1\times\cdots \times \pi_k$ is isomorphic to a quotient of the standard representation $\Sigma(\pi_1) \diamond \cdots \diamond \Sigma(\pi_k)$.\label{it:2}
  \item The representation $\pi_1\times\cdots \times \pi_k$ has a unique irreducible quotient, and that quotient is isomorphic to $\pi_1\ast \cdots\ast \pi_k$.\label{it:3}
\end{enumerate}
\end{proposition}

\begin{proof}
Clearly, \eqref{it:2} implies \eqref{it:1}. The fact that \eqref{it:3} implies \eqref{it:2} was shown in \cite[Lemma 4.2.(5)]{MR3573961}. Let us prove that  \eqref{it:1} implies \eqref{it:3}.

Suppose $\Cyc(\pi_1,\ldots,\pi_k)$ holds.  Then, $\pi_1\times\cdots\times \pi_k$ is a quotient of $\Sigma(\n)\in \Std$, for $\n\in \Mult$. By Remark \ref{rem:sum}, this means that $\pi_1\ast\cdots\ast \pi_k \preceq L(\n)$. On the other hand, since $L(\n)$ is the unique irreducible quotient of $\Sigma(\n)$, we obtain that
\[
[L(\n)] \leq [\pi_1\times\cdots \times \pi_k] \leq [\Sigma(\pi_1)\times \cdots \times \Sigma(\pi_k)] = [\Sigma(\pi_1)\diamond \cdots \diamond \Sigma(\pi_k)] = [\Sigma(\pi_1\ast\cdots\ast \pi_k)]\;,
\]
which implies that $L(\n) \preceq \pi_1\ast\cdots\ast \pi_k$. Hence, $L(\n) = \pi_1\ast\cdots\ast \pi_k$ and so $\pi_1\ast\cdots\ast \pi_k$ is the unique irreducible quotient of $\Sigma(\n)$ and hence of  $\pi_1\times\cdots\times \pi_k$.

%
%

\end{proof}
\begin{remark}
For given $\pi,\pi'\in \Irr$, we have therefore $\Cyc(\pi,\pi')$ if and only if ${\rm SSA}(\m,\m')$ in the sense of \cite[Definition 2]{LM-conj}, where $\m$ and $\m'$ are the Zelevinsky multisegments corresponding respectively to $\pi$ and $\pi'$. Through this identification, further aspects of cyclic representations are explored in \cite{LM-conj}.

We prefer our terminology in this setting because of the clear analogy to the quantum affine setting and the work of Hernandez in \cite{MR3916087}.
\end{remark}

\subsection{Cyclic representations and intertwining operators}

\begin{lemma}\label{lem:Rcyclic}
Suppose that $\Cyc(\pi_1,\pi_2)$ holds for $\pi_1,\pi_2\in \Irr$. Then, the image of $R_{\pi_1,\pi_2}$ is an irreducible representation given by $\pi_1\ast \pi_2$.
Moreover, up to a scalar, $R_{\pi_1,\pi_2}$ is the only non-zero intertwining operator from $\pi_1\times \pi_2$ to $\pi_2\times \pi_1$.
\end{lemma}
\begin{proof}
By Proposition \ref{prop:cyclic}, $\pi_1\ast \pi_2$ is the unique irreducible quotient of $\pi_1\times \pi_2$. Through Proposition \ref{prop:GK}, we also know that $\pi_1\ast\pi_2$ is the unique irreducible subrepresentation of $\pi_2\times \pi_1$.

Now, suppose that $T:\pi_1\times \pi_2\to \pi_2\times \pi_1$ is a non-zero intertwining operator. We write $K = \Ima T$. Then, $K$ must have both a unique irreducible quotient (as a quotient of $\pi_1\times \pi_2$) and a unique irreducible subrepresentation (as a subrepresentation of $\pi_2\times\pi_1$). In particular, both the quotient and the sub-representation must be isomorphic to $\pi_1\ast\pi_2$. However, by Remark \ref{rem:sum}\eqref{it:sum2}, $\pi_1\ast\pi_2$ cannot appear in $K$ with multiplicity. Thus, $K\cong \pi_1\ast\pi_2$.

Again, by the same multiplicity-one property, we see that the kernel and (irreducible) image of $T$ are uniquely determined. It follows by Schur's lemma that $T$ is unique, up to a scalar. Since $R_{\pi_1,\pi_2}$ is non-zero the result follows.
\end{proof}

Let  $\pi_1,\ldots,\pi_k\in \Irr$. For a simple transposition $\epsilon_i = (i\,i+1)\in \Symm_k$, we denote:
\[
R^{i}=  \id \times R_{\pi_i,\pi_{i+1}}\times \id
\]
for the operator between
\[
\pi_1\times \cdots \times \pi_{i-1} \times \pi_i\times \pi_{i+1}\times \pi_{i+2}\times \cdots \times \pi_k \;\to\; \pi_1\times \cdots \times \pi_{i-1} \times \pi_{i+1}\times \pi_{i}\times \pi_{i+2}\times \cdots \times \pi_k \;.
\]

\begin{proposition}\label{prop:nontriv-intert}
Suppose that $\pi_1,\ldots,\pi_k\in \Irr$ are representations, for which $\Cyc(\pi_i,\pi_j)$ holds for all $1\leq i< j\leq k$.

Let $\omega = \epsilon_{i_t}\cdots \epsilon_{i_1}\in \Symm_k$ be a permutation given in terms of a reduced decomposition into simple transpositions.

Then, the image of the operator
\[
R: = R^{i_t}\circ \cdots\circ R^{i_1}: \pi_1\times\cdots\times \pi_k \to \pi_{\omega^{-1}(1)}\times\cdots\times \pi_{\omega^{-1}(k)}
\]
contains $\pi_1\ast\cdots\ast \pi_k$ as an irreducible subquotient. In particular, $R$ is non-zero.
\end{proposition}

\begin{proof}
We prove it by induction on the length of $\omega$, that is, $t$.

By the induction hypothesis for $\omega':= \epsilon_{i_t}\omega$, the image $R^{i_{t-1}}\circ \cdots\circ R^{i_1}$ has $\Pi:=\pi_1\ast \cdots\ast \pi_k$ as a subquotient. Thus, it is enough to prove that $\Pi$ does not appear in the kernel of $R^{i_t}$.

By Remark \ref{rem:sum}\eqref{it:sum2}, $\Pi$ appears in $\pi_{\omega'^{-1}(1)}\times\cdots\times \pi_{\omega'^{-1}(k)}$ with multiplicity $1$. Hence, it is enough to show that $\Pi$ appears in the image of $R^{i_t}$. Indeed, since the given decomposition of $\omega$ is reduced, we know that $\omega^{-1}(i_t) < \omega^{-1}(i_{t+1})$ and $\Cyc(\pi_{\omega^{-1}(i_t)},\pi_{\omega^{-1}(i_{t+1})})$ holds. Thus, by Lemma \ref{lem:Rcyclic}, $\Ima R^{i_t}$ is given as
\[
\pi_{\omega^{-1}(1)}\times \cdots \times \pi_{\omega^{-1}(i_t)}\ast \pi_{\omega^{-1}(i_{t+1})} \times\cdots \times \pi_{\omega^{-1}(k)}\;.
\]
In particular, again by Remark \ref{rem:sum}\eqref{it:sum2}, it contains $\Pi$.
\end{proof}

\section{Proof of the main theorem}
Following the terminology of \cite{MR3866895}, we say that $\pi \in \Irr$ is a \textit{square-irreducible} representation, if $\pi\times \pi$ is irreducible.
We recall \cite[Corllary 2.5]{MR3866895} that those representations can be characterized by means of standard intertwining operators: A representation $\pi\in \Irr$ is square-irreducible, if and only if, $R_{\pi,\pi}$, as an intertwining operator from the space of $\pi\times \pi$ to itself, is a scalar operator.

Our main tool for exploiting the property of square-irreducibility will be the following potent lemma that was devised by Kang-Kashiwara-Kim-Oh, and was naturalized in our setting in \cite{MR3866895}. We state a reversed version of it here.

\begin{lemma}\cite[Lemma 3.1]{MR3314831}\cite[Corollary 2.2]{MR3866895}\label{lem:kkko}
Let $\pi_i\in\Reps(G_{n_i})$, $i=1,2,3$.
Let $\sigma$ (resp., $\lambda$) be a subrepresentation of $\pi_1\times\pi_2$ (resp., $\pi_2\times\pi_3$).
Assume that $\pi_1\times\lambda\subset \sigma\times\pi_3$. Then there exists a subrepresentation $\kappa$ of $\pi_2$ such that
$\pi_1\times\kappa\subset \sigma$ and $\lambda\subset\kappa\times\pi_3$.
\end{lemma}

Theorem \ref{thm:main-sqirr} below may be viewed as a corollary of Proposition \ref{prop:nontriv-intert} and the analog of \cite[Lemma 2.6]{MR4016058} or \cite[Proposition 3.2.14]{MR3758148}, when translated to our setting. Since the suitable translation from the representation theory of quiver Hecke algebras is rather involved, we prefer to construct a native proof relying on ideas that were already naturalized in our setting through previous literature.

\begin{theorem}\label{thm:main-sqirr}
Let $\pi_1,\ldots, \pi_n\in \Irr$, such that $\Cyc(\pi_i,\pi_j)$ holds, for all $1\leq i< j\leq n$. Assume that $\pi_1,\ldots, \pi_{n-2}$ are square-irreducible representations.

Then, $\Cyc(\pi_1,\ldots,\pi_n)$ holds as well.
\end{theorem}

\begin{proof}

We reason by induction on $n$. Hence, we assume that $\Cyc(\pi_2,\ldots,\pi_n)$ holds. Let us write $\Pi= \pi_2\times \cdots \times \pi_n$. By Proposition \ref{prop:cyclic}, $\Pi$ has a unique irreducible quotient, which is isomorphic to $\Theta'= \pi_2\ast \cdots \ast \pi_n$.
Let $\eta$ denote the sub-representation of $\Pi$, which is the kernel of $\Pi \to \Theta'$. It is the unique maximal proper sub-representation of $\Pi$.

Again, from Proposition \ref{prop:cyclic}, we know that it is enough to prove that $\pi_1\times \Pi$ has a unique irreducible quotient, which is isomorphic to $\Theta= \pi_1 \ast \Theta'$. In other words, we need to show that any proper subrepresentation $\tau \subset \pi_1\times \Pi$, does not contain $\Theta$ as a subquotient.

Let $\tau \subset \pi_1\times \Pi$ be a proper sub-representation. We consider the intertwining operator
\[
R_{\pi_1,\pi_1\times \Pi}: \pi_1 \times \pi_1 \times \Pi \to \pi_1 \times \Pi \times \pi_1\;.
\]
Restricting to $\tau$, we obtain an operator $\pi_1\times \tau \to \tau \times \pi_1$, by Lemma \ref{lem:intert}\eqref{part: nonzeroR}.

By assumption, $R_{\pi_1,\pi_1}$ is a scalar operator. Hence, by Lemma \ref{lem:intert}\eqref{part: nonzeroR}, $R_{\pi_1,\pi_1\times \Pi} = \lambda( \id_{\pi_1}\times R_{\pi_1,\Pi})$, for a non-zero $\lambda\in \mathbb{C}$. This identity gives us the following inclusion of subrepresentations of $\pi_1\times \Pi \times \pi_1$:
\[
\pi_1 \times R_{\pi_1,\Pi}(\tau) \subset\tau \times \pi_1\;.
\]
At this point, we can apply Lemma \ref{lem:kkko} to obtain a subrepresentation $\kappa \subset \Pi$, such that both inclusions $R_{\pi_1,\Pi}(\tau) \subset \kappa \times  \pi_1$ and $ \pi_1\times \kappa\subset \tau$ hold.

The latter inclusion implies that $\kappa$ is a proper subrepresentation of $\Pi$, hence, $\kappa \subset \eta$. The former inclusion then implies that $R_{\pi_1,\Pi}(\tau) \subset \eta \times  \pi_1$.

Now, recall that $\Theta$ appears in $\Pi \times\pi_1$ with multiplicity one, and that it must also appear in $\Theta' \times \pi_1$. Thus, by exactness of parabolic induction $\eta \times \pi_1$, and its subrepresentation $R_{\pi_1,\Pi}(\tau)$, cannot contain a subquotient isomorphic to $\Theta$.

On the other hand, applying Proposition \ref{prop:nontriv-intert} and Lemma \ref{lem:intert}\eqref{part: nonzeroR}, we see that $R_{\pi_1,\Pi}$ must be a non-zero operator, whose image contains $\Theta$ as a subquotient. Using the multiplicity-one property again, we conclude that $\ker R_{\pi_1,\Pi}$ cannot contain $\Theta$ as a subquotient.

If both $\ker R_{\pi_1,\Pi}$ and $ R_{\pi_1,\Pi}(\tau)$ cannot contain $\Theta$ as a sub-quotient, neither can $\tau$.

\end{proof}

The following corollary generalizes \cite[Corollary 6.9]{MR3573961}.

\begin{corollary}\label{cor:main}
Let $\pi_1, \dots ,\pi_n \in \Irr$ be given. Assume that all but at most two are square irreducible.
Then, $\pi_i \times \pi_j$ is irreducible for all $1\leq i < j\leq n$, if and only if, $\pi_1  \times \dots \times \pi_n$ is irreducible.
\end{corollary}
\begin{proof}
One direction is obvious by exactness of parabolic induction. Let us prove the converse.
Recall that, for all $i< j$, since $\pi_i\times \pi_j\in \Irr$, we have $\pi_j\times \pi_i\cong \pi_i\times \pi_j$. In particular, we can assume that $\pi_1, \dots, \pi_{n-2}$ are square irreducible and $\Cyc(\pi_i,\pi_j)$ holds, for $i\neq j$. Thus, by Theorem \ref{thm:main-sqirr}, both $\Lambda = \pi_1\times\cdots \times \pi_n$ and $\Lambda'= \pi_1^\vee \times\cdots \times \pi_n^\vee$ are cyclic.

Therefore, $\Pi=\pi_1\ast\cdots \ast \pi_n$ appears as a unique irreducible quotient of  $\Lambda$ and $\Pi^\vee=\pi_1^\vee\ast\cdots \ast \pi_n^\vee$ appears as a unique irreducible quotient of $\Lambda'$. By applying the contragredient functor, we find that $\Pi$ appears also as a unique irreducible subrepresentation of $\Lambda$.
By Remark \ref{rem:sum}\eqref{it:sum2}, this forces $\Lambda$ to be irreducible.

\end{proof}



\bibliographystyle{amsalpha}
\bibliography{../Bibfiles/all,../propo2}

\end{document}